\documentclass{mrlart7}
  \def\volno{18}
  \def\yrno{2011}
  \def\issueno{00}
  \def\lpageno{100NN} %temporary numbers
\usepackage{amsmath,amssymb,smfthm,stmaryrd,epic,enumerate}
\usepackage[frenchb,english]{babel}
\NumberTheoremsIn{subsection}\SwapTheoremNumbers

\begin{document}

\newcommand{\REMARK}[1]{\marginpar{\tiny #1}}
\newtheorem{thm}{Theorem}[section]
\newtheorem{lemma}[thm]{Lemma}
\newtheorem{corol}[thm]{Corollary}
\newtheorem{propo}[thm]{Proposition}
\newtheorem{defin}[thm]{Definition}
\newtheorem{Remark}[thm]{Remark}
\numberwithin{equation}{section}

\newtheorem{notas}[thm]{Notations}
\newtheorem{nota}[thm]{Notation}
\newtheorem{defis}[thm]{Definitions}
\newtheorem*{thm*}{Theorem}

\def\Tm{{\mathbb T}}
\def\Am{{\mathbb A}}
\def\Fm{{\mathbb F}}
\def\Qm{{\mathbb Q}}
\def\Zm{{\mathbb Z}}
\def\Cm{{\mathbb C}}
\def\Rm{{\mathbb R}}
\def\IC{{\mathcal I}}
\def\OC{{\mathcal O}}

\def \art{\mathop{\mathrm{Art}}\nolimits}
\def \speh {\mathop{\mathrm{Speh}}\nolimits}
\def \groth {\mathop{\mathrm{Groth}}\nolimits}
\def \val {\mathop{\mathrm{val}}\nolimits}
\def \det {\mathop{\mathrm{det}}\nolimits}
\def \spec {\mathop{\mathrm{Spec}}\nolimits}
\def \diag {\mathop{\mathrm{diag}}\nolimits}
\def \lie {{\mathop{\mathrm{Lie}}\nolimits}}
\def \ind {\mathop{\mathrm{ind}}\nolimits}
\def \st {{\mathop{\mathrm{St}}\nolimits}}
\def \gr {{\mathop{\mathrm{gr}}\nolimits}}
\def \Spl {{\mathop{\mathrm{Spl}}\nolimits}}
\def \bad{\mathop{\mathrm{Bad}}\nolimits}
\def \unr{\mathop{\mathrm{Unr}}\nolimits}
\def \rem{{\noindent\textit{Remark:~}}}

\def \hi{\HC}

\setcounter{secnumdepth}{3} \setcounter{tocdepth}{3}

\def \Fil{\mathop{\mathrm{Fil}}\nolimits}

\let \oo=\infty

\newcommand\atopp[2]{\genfrac{}{}{0pt}{}{#1}{#2}}

\title{Torsion classes in the cohomology of KHT Shimura varieties}

\author{Boyer Pascal}

\address{Universit\'e Paris 13, Sorbonne Paris Cit\'e \\
LAGA, CNRS, UMR 7539\\ 
F-93430, Villetaneuse (France) \\
PerCoLaTor: ANR-14-CE25}

\email{boyer@math.univ-paris13.fr}

%\thanks{The author thanks the ANR PerCoLaTor 14-CE25.}

\begin{abstract}
A particular case of Bergeron-Venkatesh's conjecture predicts 
that torsion classes in the cohomology of Shimura varieties are rather rare.
According to this and for Kottwitz-Harris-Taylor type of Shimura varieties, we first associate to each 
such torsion class an infinity of irreducible automorphic representations in characteristic zero, 
which are pairwise non isomorphic and weakly congruent in the sense of \cite{vigneras-langlands}.
Then, using completed cohomology, we construct torsion classes in regular weight 
and then deduce explicit examples of such automorphic congruences.

\end{abstract}

%% Classification mathÈmatique  (2010)
%\subjclass{11F70, 11F80, 11F85, 11G18, 20C08}


\maketitle

\pagestyle{headings} \pagenumbering{arabic}

\tableofcontents
%
%\mainmatter
%
\renewcommand{\theequation}{\arabic{section}.\arabic{subsection}.\arabic{thm}}

\section*{Introduction}
\renewcommand{\thethm}{\arabic{section}.\arabic{subsection}.\arabic{thm}}
%\backmatter

Let $F=EF^+$ be a CM field and $B/F$ a central division algebra of dimension $d^2$ 
equipped with an involution of
second kind: we can then define a group of similitudes $G/\Qm$ as explained in \S \ref{para-geo}
whose unitary associated group has signature $(1,d-1)$ at one real place and $(0,d)$ at the others.
We denote $X_{I,\eta} \rightarrow \spec F$ the Shimura variety of 
Kottwitz-Harris-Taylor type associated to $G/\Qm$ and an open compact subgroup $I$. 
For a fixed prime number $l$, consider the set $\Spl(I)$ of places $v$ of $F$ over
a prime number $p \neq l$ such that 
\begin{itemize}
\item $p=uu^c$ is split in the quadratic imaginary extension $E/Q$,

\item $G(\Qm_p)$ is split, i.e. of the following shape $\Qm_p^\times \times \prod_{w |u} (B^{op}_w)^\times$,

\item the local component at $p$ of $I$, is maximal,

\item $v|u$ and $B_v^\times \simeq GL_d(F_v)$.
\end{itemize}
Given an irreducible algebraic representation $\xi$ of $G$ which gives a $\overline \Zm_l$-local system
$V_{\xi, \overline{\mathbb Z}_l}$ over $X_{I,\bar \eta}$, if we believe in the general conjectures of \cite{BV}, and as the defect
equals $0$, asymptotically as the level $I$ increases, the torsion cohomology classes in 
$H^i(X_{I,\bar \eta},V_{\xi, \overline{\mathbb Z}_l})$ are rather rare.
In this direction, the main theorem of \cite{boyer-imj} gives  
a way to cancel this torsion by imposing for a place $v \in \Spl(I)$ that the multiset of modulo $l$ Satake parameters 
does not contain a sub-multiset of the form $\{ \alpha,q_v \alpha \}$  where $q_v$ is the cardinality of the residual field 
$\kappa(v)$ at $v$:
put another way, for a torsion cohomology class to exist, the associated set of modulo $l$
Satake parameters should, at every place $v \in \Spl(I)$, contain a subset of the shape $\{ \alpha, q_v \alpha \}$. 
In this paper, in the opposite direction, we are interested in cohomology torsion classes and their arithmetic
applications: the main result is corollary \ref{coro-principal} which can be stated as follow.

\begin{thm*} 
Let $\mathfrak m$ be a maximal ideal of some unramified Hecke algebra, associated to some
non trivial torsion class in the cohomology of $X_{I,\bar \eta}$ with coefficients in
$V_{\xi, \overline{\mathbb Z}_l}$: we denote by $i$ the greatest integer such that the torsion of
$H^{d-1-i}(X_{I,\bar \eta},V_{\xi, \overline{\mathbb Z}_l})_{\mathfrak m}$ is non trivial. There exists then a set
$$\Bigl \{ \Pi(v):~ v \in \Spl(I) \Bigr \}$$ 
of irreducible automorphic $\xi$-cohomological representations such that for all $w \in \Spl(I)$
distinct from $v$, the local component at $w$ of $\Pi(v)$ is unramified, its modulo $l$
Satake parameters being given by $\mathfrak m$. On the other hand,
$\Pi(v)$ is ramified at $v$ and more precisely, it's isomorphic to a representation of the following shape
$$\st_{i+2}(\chi_{v,0}) \times \chi_{v,1} \times \cdots \times \chi_{v,d-i-2}$$
where $\chi_{v,0},\cdots,\chi_{v,d-i-2}$ are unramified characters of $F_v$ and, see notation \ref{nota-ind}, 
the symbol $\times$ means normalized parabolic induction.
\end{thm*}

\noindent \textit{Remarks:} \\
(a) As expected from the main result of \cite{boyer-imj}, for a $\mathfrak m$-torsion cohomology class to exist, 
the Satake parameters modulo $l$ at each $v \in \Spl(I)$ must contain a chain 
$\{ \alpha, q_v \alpha,\cdots ,q_v^{a-1} \alpha \}$ with $a \geq 2$. Moreover if you want to have torsion at distance $i>0$ from
middle degree then these sets of Satake parameters, for each $v \in \Spl(I)$, have to contain such a chain 
with length at least $i+2$.
\\
(b) As showed in \cite{scholze-cara}, torsion cohomology classes of compact Shimura varieties raise in characteristic $0$
and one may asks about the precise level. The previous theorem tells you that, for a torsion cohomology class in
$H^{d-1-i}$, beside the condition about the existence for every $v \in \Spl(I)$, of a chain $\{ \alpha, q_v \alpha, \cdots, \alpha^{i+1} \}$
of length $i+2$ inside the multiset of modulo $l$ Satake parameters, to raise in characteristic $0$, it suffices
to raise the level at one $v \in \Spl(I)$ to the subgroup\footnote{It's well known that for $t_1 \geq t_2 \geq \cdots \geq t_a$, 
$\st_{t_1}(\chi_{v,1}) \times \cdots\times \st_{t_a}(\chi_{v,a})$ has non trivial invariant vectors under 
$\ker (GL_{t_1+\cdots+t_a}(\OC_v) \longrightarrow 
P_{s_1,\cdots,s_b}(\kappa(v)))$ if $(s_1 \geq \cdots \geq s_b)$ is the conjugated partition associated to $(t_1 \geq \cdots \geq t_a)$ and
this subgroup is, up to conjugacy, the minimal one among groups of this form.}
$$\ker \bigl ( GL_d(\OC_v) \longrightarrow P_{d-i-1,d-i,\cdots,d}(\kappa(v)) \bigr )$$
where $P_{d-i-1,1,\dots,1}$ is the standard
parabolic subgroup with Levi $GL_{d-i-1} \times GL_1 \times \cdots \times GL_1$ of $GL_d$. In particular the more
you go away from middle degree, to raise your torsion cohomology class in characteristic zero, the more you need, a priori, to
raise the level.
\\
(c) For $v \neq w \in \Spl(I)$, the representations $\Pi(v)$ and $\Pi(w)$ are not isomorphic but are
weakly congruent in the sense of \S 3 \cite{vigneras-langlands}, i.e. they share the same modulo $l$
Satake parameters at each place you can define Satake parameters and so in particular at all places of $\Spl(I)- \{ v,w \}$. 
\\
(d) In a recent preprint \cite{boyer-mazur}, this result is the main tool for generalizing 
Mazur's principle for $G$ i.e. to do level lowering for automorphic representations of $G$. Roughly the idea goes like that:
start from a maximal ideal $\mathfrak m$ which raises in characteristic $0$ for some level $I$ ramified at a fixed place $w$.
The tricky part is to find some good hypothesis to force the existence of a non trivial torsion cohomology class with level $I'=I^w I'_w$ 
with $I_w \subsetneq I'_w$.
Then use the previous theorem at a place $v \neq w$ to raise again in characteristic zero: then you succeed to 
lower the level at the place $w$ but you might have to increase the level at $v$.

We are then lead to the
construction of such torsion cohomolgy classes. In \S \ref{para-torsion}, we investigate this question 
with the help of the notion of completed cohomology
$$\widetilde{H}^i_{I^l}(V_{\xi,\OC}):=\lim_{\atopp{\longleftarrow}{n}}
\lim_{\atopp{\longrightarrow}{I_l}} H^i(X_{I^lI_l},V_{\xi,\OC/\lambda^n}).$$
As they are independent of the choice of the weight $\xi$,
\begin{itemize}
\item by taking $\xi$ the trivial representation and using the Hochschild-Serre spectral sequence which,
starting from the completed cohomology, computes the cohomology at a finite level, we can show
that for each divisor $s$ of $d$, the free quotient of $\widetilde H^{d-s}_{I^l}$ is non trivial provided
that the level $I^l$ outside $l$ is small enough: in fact here, we just prove an imprecise
version of this fact, see proposition \ref{prop-completee-nonnul}.

\item Then taking $\xi$ regular and as the free quotient of the finite cohomology outside the middle
degree are trivial, we observe that, for each divisor $s$ of $d$, we can find torsion classes
in level $I_l$ so that by considering their reduction modulo $l^n$ for varying $n$, and through the process of
completed cohomology i.e. taking first the limit on $I_l$ and then on $l^n$, these classes
organize themselves in some torsion free classes.
\end{itemize}

Automorphic congruences can be useful to construct non trivial elements in Selmer
groups. For this one need to obtain such congruences between tempered and non tempered
automorphic representations of same weight, which is not exactly the case here. In some forthcoming xork, we should be able
to do so using 
\begin{itemize}
\item either torsion classes in the cohomology of Harris-Taylor's local system constructed in
\cite{boyer-aif} and associated to non tempered automorphic representation. The idea is then to prove 
that this torsion classes give torsion classes in the cohomology of the whole Shimura variety
and use the main theorem of this paper.

\item or the cancellation of completed cohomology groups $\widetilde{H}^i_{I^l}(V_{\xi,\OC})$ when $i\geq d$,
see \cite{scholze-torsion} proposition IV.2.2.
\end{itemize}

%
%\rem en ce qui concerne le cas irrÈgulier,
%dans un prochain papier nous construirons, sous certaines hypothËses, 
%des classes de cohomologie de torsion donnant des ÈlÈments non nuls de certains groupes de Selmer.

%\renewcommand{\thesmfthm}{\arabic{section}.\arabic{subsection}.\arabic{smfthm}}

\section{Notations and background}

\subsection{Induced representations}

Consider a local fiel $K$ with its absolute value $| - |$: let $q$ denote the cardinal of its residual field.
For a representation $\pi$ of $GL_d(K)$ and $n \in \frac{1}{2} \Zm$, set 
$$\pi \{ n \}:= \pi \otimes q^{-n \val \circ \det}.$$

\begin{notas} \label{nota-ind}
For $\pi_1$ and $\pi_2$ representations of respectively $GL_{n_1}(K)$ and
$GL_{n_2}(K)$, we will denote by
$$\pi_1 \times \pi_2:=\ind_{P_{n_1,n_1+n_2}(K)}^{GL_{n_1+n_2}(K)}
\pi_1 \{ \frac{n_2}{2} \} \otimes \pi_2 \{-\frac{n_1}{2} \},$$
the normalized parabolic induced representation where for any sequence
$\underline r=(0< r_1 < r_2 < \cdots < r_k=d)$, we write $P_{\underline r}$ for 
the standard parabolic subgroup of $GL_d$ with Levi
$$GL_{r_1} \times GL_{r_2-r_1} \times \cdots \times GL_{r_k-r_{k-1}}.$$ 
\end{notas}

Remind that an irreducible representation is called supercuspidal if it's not a subquotient
of some proper parabolic induced representation.

\begin{defin} \label{defi-rep} (see \cite{zelevinski2} \S 9 and \cite{boyer-compositio} \S 1.4)
Let $g$ be a divisor of $d=sg$ and $\pi$ an irreducible cuspidal representation of $GL_g(K)$. 
The induced representation
$$\pi\{ \frac{1-s}{2} \} \times \pi \{ \frac{3-s}{2} \} \times \cdots \times \pi \{ \frac{s-1}{2} \}$$ 
holds a unique irreducible quotient (resp. subspace) denoted $\st_s(\pi)$ (resp.
$\speh_s(\pi)$); it's a generalized Steinberg (resp. Speh) representation.
\end{defin}

\rem from a galoisian point of view through the local Langlands correspondence, the
representation $\speh_s(\pi)$ matches to the direct sum $\sigma(\frac{1-s}{2}) \oplus \cdots \oplus
\sigma(\frac{s-1}{2})$ where $\sigma$ matches to $\pi$. \emph{More generally} for
$\pi$ any irreducible representation of $GL_g(K)$ associated to $\sigma$ by the local Langlands
correspondence, we will denote $\speh_s(\pi)$ the representation of
$GL_{sg}(K)$ matching, through the local Langlands correspondence,  
$\sigma(\frac{1-s}{2}) \oplus \cdots \oplus \sigma(\frac{s-1}{2})$.

\begin{defin} A smooth $\overline \Qm_l$-representation of finite length $\pi$ of $GL_d(K)$ is said
\emph{entire} if there exist a finite extension $E/ \Qm_l$ contained in
$\overline \Qm_l$, with ring of integers  $\OC_E$, and a $\OC_E$-representation $L$ of 
$GL_d(K)$, which is a free $\OC_E$-module, such that $\overline \Qm_l \otimes_{\OC_E} L \simeq \pi$
and $L$ is a $\OC_E $ $GL_n(K)$-module of finite type. Let 
$\kappa_E$ the residual field of $\OC_E$, we say that 
$$\overline \Fm_l \otimes_{\kappa_E} \kappa_E \otimes_{\OC_E} L$$ 
is the modulo $l$ reduction of $L$. 
\end{defin}

\rem the \textit{Brauer-Nesbitt principle} asserts that the semi-simplification of  
$\overline \Fm_l \otimes_{\OC_E} L$ is a finite length $\overline \Fm_l$-representation of $GL_d(K)$ 
which is independent of the choice of $L$. Its image in the Grothendieck group will be denoted
$r_l(\pi)$ and called the \textit{modulo $l$ reduction of $\pi$}.

\noindent \textit{Example}:  from \cite{vigneras-induced} V.9.2 or \cite{dat-jl} \S 2.2.3, we know that
the modulo $l$ reduction of $\speh_s(\pi)$ is irreducible.

\subsection{Geometry of KHT Shimura varieties}
\label{para-geo}

Let $F=F^+ E$ be a CM field where $E/\Qm$ is quadratic imaginary and $F^+/\Qm$
totally real with a fixed real embedding $\tau:F^+ \hookrightarrow \Rm$. For a place $v$ of $F$,
we will denote
\begin{itemize}
\item $F_v$ the completion of $F$ at $v$,

\item $\OC_v$ the ring of integers of $F_v$,

\item $\varpi_v$ a uniformizer,

\item $q_v$ the cardinal of the residual field $\kappa(v)=\OC_v/(\varpi_v)$.
\end{itemize}
Let $B$ be a division algebra with center $F$, of dimension $d^2$ such that at every place $x$ of $F$,
either $B_x$ is split or a local division algebra. 

\begin{nota} Let denote $\bad$ the set of places of $F$ such that for any $w \not \in \bad$, we have
$B_w^\times \simeq GL_d(F_w)$.
\end{nota}

Further we assume $B$ provided with an involution of
second kind $*$ such that $*_{|F}$ is the complex conjugation. For any
$\beta \in B^{*=-1}$, denote $\sharp_\beta$ the involution $x \mapsto x^{\sharp_\beta}=\beta x^*
\beta^{-1}$ and $G/\Qm$ the group of similitudes, denoted $G_\tau$ in \cite{h-t}, defined for every
$\Qm$-algebra $R$ by 
$$
G(R)  \simeq   \{ (\lambda,g) \in R^\times \times (B^{op} \otimes_\Qm R)^\times  \hbox{ such that } 
gg^{\sharp_\beta}=\lambda \}
$$
with $B^{op}=B \otimes_{F,c} F$. 
If $x$ is a place of $\Qm$ split $x=yy^c$ in $E$ then 
\addtocounter{thm}{1}
\begin{equation} \label{eq-facteur-v}
G(\Qm_x) \simeq (B_y^{op})^\times \times \Qm_x^\times \simeq \Qm_x^\times \times
\prod_{z_i} (B_{z_i}^{op})^\times,
\end{equation}
where, identifying places of $F^+$ over $x$ with places of $F$ over $y$,
$x=\prod_i z_i$ in $F^+$.

\noindent \textbf{Convention}: for $x=yy^c$ a place of $\Qm$ split in $E$ and $z$ a place of $F$ over $y$
as before, we shall make throughout the text, the following abuse of notation by denoting 
$G(F_z)$ in place of the factor $(B_z^{op})^\times$ in the formula (\ref{eq-facteur-v})
as well as
$$G(\Am^z):=G(\Am^x) \times \bigl (\Qm_x^\times \times \prod_{z_i \neq z} (B_{z_i}^{op})^\times \bigr ).$$

In \cite{h-t}, the author justify the existence of some $G$ like before such that moreover
\begin{itemize}
\item if $x$ is a place of $\Qm$ non split in $E$ then $G(\Qm_x)$ is quasi split;

\item the invariants of $G(\Rm)$ are $(1,d-1)$ for the embedding $\tau$ and $(0,d)$ for the others.
\end{itemize}

As in  \cite{h-t} bottom of page 90, a compact open subgroup $U$ of $G(\Am^\oo)$ is said 
\emph{small enough}
if there exists a place $x$ such that the projection from $U^v$ to $G(\Qm_x)$ does not contain any 
element of finite order except identity.

\begin{nota}
Denote $\IC$ the set of open compact subgroup small enough of $G(\Am^\oo)$.
For $I \in \IC$, write $X_{I,\eta} \longrightarrow \spec F$ the associated
Shimura variety of Kottwitz-Harris-Taylor type.
\end{nota}

From now on, we fix a prime number $l$ unramified in $E$.

\begin{defin} \label{defi-spl}
Define $\Spl$ the set of  places $v$ of $F$ such that $p_v:=v_{|\Qm} \neq l$ is split in $E$ and
$B_v^\times \simeq GL_d(F_v)$.  For each $I \in \IC$, write
$\Spl(I)$ the subset of $\Spl$ of places which doesn't divide the level $I$.
\end{defin}

\rem for every $v \in \Spl$, the variety $X_{I,\eta}$ has a projective model $X_{I,v}$ over $\spec \OC_v$
with special fiber $X_{I,s_v}$. For $I$ going through $\IC$, the projective system $(X_{I,v})_{I\in \IC}$ 
is naturally equipped with an action of $G(\Am^\oo) \times \Zm$ such that 
$w_v$ in the Weil group $W_v$ of $F_v$ acts by $-\deg (w_v) \in \Zm$,
where $\deg=\val \circ \art^{-1}$ and $\art^{-1}:W_v^{ab} \simeq F_v^\times$ is Artin's isomorphism
which sends geometric Frobenius to uniformizers.

\begin{notas} (see \cite{boyer-invent2} \S 1.3)
For $I \in \IC$, the Newton stratification of the geometric special fiber $X_{I,\bar s_v}$ is denoted
$$X_{I,\bar s_v}=:X^{\geq 1}_{I,\bar s_v} \supset X^{\geq 2}_{I,\bar s_v} \supset \cdots \supset 
X^{\geq d}_{I,\bar s_v}$$
where $X^{=h}_{I,\bar s_v}:=X^{\geq h}_{I,\bar s_v} - X^{\geq h+1}_{I,\bar s_v}$ is an affine 
scheme\footnote{see for example \cite{ito2}}, smooth of pure dimension $d-h$ built up by the geometric 
points whose connected part of its Barsotti-Tate group is of rank $h$.
For each $1 \leq h <d$, write
$$i_{h+1}:X^{\geq h+1}_{I,\bar s_v} \hookrightarrow X^{\geq h}_{I,\bar s_v}, \quad
j^{\geq h}: X^{=h}_{I,\bar s_v} \hookrightarrow X^{\geq h}_{I,\bar s_v}.$$
%and $j^{=h}=i_h \circ j^{\geq h}$.
\end{notas}

\subsection{Cohomology groups over $\overline \Qm_l$}

Let us begin with some known facts about irreducible algebraic representations of $G$,
see for example \cite{h-t} p.97. Let $\sigma_0:E \hookrightarrow
\overline{\Qm}_l$ be a fixed embedding and et write $\Phi$ the set of embeddings 
$\sigma:F \hookrightarrow \overline \Qm_l$ whose restriction to $E$ equals $\sigma_0$.
There exists then an explicit bijection between irreducible algebraic representations $\xi$ of $G$ 
over $\overline \Qm_l$ and $(d+1)$-uples
$\bigl ( a_0, (\overrightarrow{a_\sigma})_{\sigma \in \Phi} \bigr )$
where $a_0 \in \Zm$ and for all $\sigma \in \Phi$, we have $\overrightarrow{a_\sigma}=
(a_{\sigma,1} \leq \cdots \leq a_{\sigma,d} )$.

For $K \subset \overline \Qm_l$ a finite extension of $\Qm_l$ such that the representation
$\iota^{-1} \circ \xi$ of highest weight
$\bigl ( a_0, (\overrightarrow{a_\sigma})_{\sigma \in \Phi} \bigr )$,
is defined over $K$, write $W_{\xi,K}$ the space of this representation and $W_{\xi,\OC}$
a stable lattice under the action of the maximal open compact subgroup $G(\Zm_l)$, 
where $\OC$ is the ring of integers of $K$ with uniformizer $\lambda$.

\rem if $\xi$ is supposed to be $l$-small, in the sense that for all $\sigma \in \Phi$ and all
$1 \leq i < j \leq n$ we have $0 \leq a_{\tau,j}-a_{\tau,i} < l$, then such a stable lattice is unique
up to a homothety.

\begin{nota} \label{nota-Vxi}
We will denote $V_{\xi,\OC/\lambda^n}$ the local system on $X_\IC$ as well as
$$V_{\xi,\OC}=\lim_{\atopp{\longleftarrow}{n}} V_{\xi,\OC/\lambda^n} \quad \hbox{and} \quad
V_{\xi,K}=V_{\xi,\OC} \otimes_{\OC} K.$$
For $\overline \Zm_l$ and $\overline \Qm_l$ version, we will write respectively
$V_{\xi,\overline \Zm_l}$ and $V_{\xi,\overline \Qm_l}$. 
%We'll add the symbol $\xi$ on asheaf to indicate its torsion by $V_{\xi,\overline \Zm_l}$: for example $HT_\xi(\pi_v,\Pi_t):=HT(\pi_v,\Pi_t) \otimes V_{\xi,\overline \Zm_l}$.
\end{nota}

\rem the representation $\xi$ is said \emph{regular} if its parameter
$\bigl ( a_0, (\overrightarrow{a_\sigma})_{\sigma \in \Phi} \bigr )$ 
verify for all  $\sigma \in \Phi$ that $a_{\sigma,1} < \cdots < a_{\sigma,d}$.

\begin{defin}
An irreducible automorphic representation $\Pi$ is said $\xi$-cohomological if there exists
an integer $i$ such that
$$H^i \bigl ( ( \lie ~G(\Rm)) \otimes_\Rm \Cm,U,\Pi_\oo \otimes \xi^\vee \bigr ) \neq (0),$$
where $U$ is a maximal open compact subgroup modulo the center of $G(\Rm)$.
\end{defin}

For $\Pi$ an automorphic irreducible representation $\xi$-cohomological of $G(\Am)$,
then, see for example lemma 3.2 of \cite{boyer-aif}, for each $v \in \Spl$,
the local component $\Pi_v$ is isomorphic to some $\speh_s(\pi_v)$ where $\pi_v$
is an irreducible non degenerate representation and $s \geq 1$ an integer which is independent of
the place $v \in \Spl$.

\begin{defin} \label{defi-parametre}
The integer $s$ mentioned above is called the degeneracy depth of $\Pi$.
\end{defin}

From now on, we fix $v \in \Spl$.

%
%\begin{nota} For $\pi_v$ an irreducible cuspidal representation of $GL_g(F_v)$, write 
%$$s_\xi(\pi_v)$$ 
%for the biggest integer $s$ such that there exists an automorphic $\xi$-cohomological representation
%$\Pi$ such that its local component at $v$ is isomorphic to some $\speh_{s}(\pi'_{v}) \times ?$ where
%$\pi'_v$ is inertially equivalent to $\pi_v$ and $?$ is an unknown representation of
%$GL_{d-sg_{-1}}(F_v)$ we don't try to precise.
%\end{nota}

\begin{nota} \label{nota-hixi}
For $1 \leq h \leq d$, let us denote $\IC_v$ the set of open compact subgroups of the following shape
$$U_v(\underline m):= U_v(\underline m^v) \times K_v(m_1),$$
where $K_v(m_1)=\ker \bigl ( GL_{d}(\OC_v) \longrightarrow GL_{d}(\OC_v/ (\varpi_v^{m_1})) \bigr )$.
The notation $[H^i(h,\xi)]$ (resp. $[H^i_!(h,\xi)]$) means the image of 
$$\lim_{\atopp{\longrightarrow}{I \in \IC_v}} H^i(X_{I,\bar s_v}^{\geq h}, V_{\xi,\overline \Qm_l}) 
\qquad \hbox{resp. }
\lim_{\atopp{\longrightarrow}{I \in \IC_v}} H^i_c(X_{I,\bar s_v}^{=h},V_{\xi,\overline \Qm_l})$$ 
in the Grothendieck group $\groth(v)$ of admissible representations of
$G(\Am^{\oo,v}) \times GL_{d}(F_v) \times \Zm$.
\end{nota}

\rem recall that the action of $\sigma \in W_v$ on these $GL_{d}(F_v) \times \Zm$-modules 
is given by those of $-\deg \sigma \in \Zm$.

\begin{nota}
For $\Pi^{\oo,v}$ an irreducible representation of $G(\Am^{\oo,v})$, let denote $\groth(h)\{ \Pi^{\oo,v} \}$ 
the subgroup of $\groth(v)$ generated by representations of the shape 
$\Pi^{\oo,v} \otimes \Psi_{v} \otimes \zeta$ where $\Psi_{v}$ (resp. $\zeta$) is any irreducible
representation of  $GL_{d}(F_v)$ (resp. of $\Zm$). We will denote then
$$[H^i(h,\xi)]\{ \Pi^{\oo,v} \}$$
the projection of $[H^i(h,\xi)]$ on this direct factor.
\end{nota}

We write
$$[H^i(h,\xi)]\{ \Pi^{\oo,v} \}=\Pi^{\oo,v} \otimes \Bigl ( \sum_{\Psi_v,\xi} m_{\Psi_v,\zeta}(\Pi^{\oo,v}) \Psi_v 
\otimes \zeta \Bigr ),$$ 
where $\Psi_v$ (resp. $\xi$) goes through irreducible admissible representations of $GL_d(F_v)$,
(resp. of $\Zm$ which can be considered as an unramified representation of $W_v$).

\noindent \textit{Remark}: recall, cf. \cite{boyer-compositio} where all these cohomology groups
are explicitly computed, that if for some $h$ and $i$,
$[H^i(h,\xi)]\{ \Pi^{\oo,v} \} \neq (0)$ then $\Pi^{\oo,v}$ is the component of a automorphic
$\xi$-cohomological representation.

\begin{propo} \label{prop-temperee}
Let $\Pi$ be an automorphic irreducible tempered representation $\xi$-cohomological.
\begin{itemize}
\item[(i)] For all $h=1,\cdots,d$ and all $i \neq d-h$,
$$[H^i(h,\xi)]\{ \Pi^{\oo,v} \} \qquad \hbox{and} \qquad [H^i_!(h,\xi)]\{ \Pi^{\oo,v} \}$$
are trivial.

\item[(ii)] If $[H^{d-h}(h,\xi)]\{ \Pi^{\oo,v} \}$ (resp. $[H^{d-h}_!(h,\xi)]\{ \Pi^{\oo,v} \}$) has non trivial
invariants under the action of $GL_d(\OC_v)$ then the local component $\Pi_v$ of $\Pi$ at $v$ 
is isomorphic to a representation of the following shape
$$\st_r(\chi_{v,0}) \times \chi_{v,1} \times \cdots \times \chi_{v,r}$$
where $\chi_{v,0},\cdots,\chi_{v,t}$ are unramified characters and $r =h$ (resp. $r \geq h$).
\end{itemize}
\end{propo}

\begin{proof}
(i) This is exactly proposition 1.3.9 of \cite{boyer-imj}.

(ii) The result for $H^i(h,\xi)$ is a particular case of proposition 3.6 of \cite{boyer-aif} (which proposition
follows directly from proposition 3.6.1 of \cite{boyer-compositio}) for the constant local system, i.e.
when $\pi_v$ is the trivial representation and $s=1$.

Concerning the cohomology with compact supports, we can use either proposition 3.12 of \cite{boyer-aif}
or the description, given by corollary 5.4.1 of \cite{boyer-invent2}, of this extension by zero
in terms of local systems on Newton strata with indices $h' \geq h$.
\end{proof}

\begin{propo} \label{prop-non-temperee} (see \cite{boyer-compositio} theorem 4.3.1)
Let $\Pi$ be an automorphic irreducible representation $\xi$-cohomological with depth of degeneracy
$s>1$. Then for $\xi$ the trivial character, 
$[H^{d-2h+s}_!(h,\xi)]\{ \Pi^{\oo,v} \}$ is non trivial.
\end{propo}

\rem if $\xi$ is a regular parameter then the depth of degeneracy of any irreducible automorphic
representation $\xi$-cohomological is necessary equal to $1$. In particular
theorem 4.3.1 of \cite{boyer-compositio} is compatible with the classical result saying that
for a regular $\xi$, the cohomology of the Shimura variety $X_I$ with coefficients
in $V_{\xi,\overline \Qm_l}$, is concentrated in middle degree.

\subsection{Hecke algebras}

Consider the following set $\unr(I)$ which is the union of
\begin{itemize}
\item places $q \neq l$ of $\Qm$ inert in $E$ not below a place of $\bad$ and where $I_q$ is maximal;

\item places $w \in \Spl(I)$.
\end{itemize}

\begin{nota} \label{nota-spl2}
For $I \in \IC$ a finite level, write
$$\Tm_I:=\prod_{x \in \unr(I)} \Tm_{x}$$
where for $x$ a place of $\Qm$ (resp. $x \in \Spl(I)$), 
$\Tm_{x}$ is the unramified Hecke algebra of $G(\Qm_x)$ (resp. of $GL_d(F_x)$) over
$\overline \Zm_l$.
\end{nota}

\noindent \textit{Example}:  for $w \in \Spl(I)$, we have
$$\Tm_{w}=\overline \Zm_l \bigl [T_{w,i}:~ i=1,\cdots,d \bigr ],$$
where $T_{w,i}$ is the characteristic function of
$$GL_d(\OC_w) \diag(\overbrace{\varpi_w,\cdots,\varpi_w}^{i}, \overbrace{1,\cdots,1}^{d-i} ) 
GL_d(\OC_w) \subset  GL_d(F_w).$$
More generally, the Satake isomorphism identifies $\Tm_x$ with $\overline \Zm_l[X^{un}(T_x)]^{W_x}$
where
\begin{itemize}
\item $T_x$ is a split torus,

\item $W_x$ is the spherical Weyl group

\item and $X^{un}(T_x)$ is the set of $\overline \Zm_l$-unramified characters of $T_x$.
\end{itemize}

Consider a fixed maximal ideal $\mathfrak m$ of $\Tm_I$ and for every $x \in \unr(I)$ let denote
$S_{\mathfrak m}(x)$ be the multi-set\footnote{A multi-set is a set with multiplicities.} 
of modulo $l$ Satake parameters at $x$ associated to $\mathfrak m$.

\noindent \textit{Example}: 
for every $w \in \Spl(I)$, the multi-set of Satake parameters at $w$ corresponds to the roots of
the Hecke polynomial
$$P_{\mathfrak{m},w}(X):=\sum_{i=0}^d(-1)^i q_w^{\frac{i(i-1)}{2}} \overline{T_{w,i}} X^{d-i} \in \overline 
\Fm_l[X]$$
i.e.
$$S_{\mathfrak{m}}(w) := \bigl \{ \lambda \in \Tm_I/\mathfrak m \simeq \overline \Fm_l \hbox{ such that }
P_{\mathfrak{m},w}(\lambda)=0 \bigr \}.$$

To each Hecke polynomial $P_{\mathfrak m,x}(X)$ at $x \in \unr(I)$, one can associate its reciprocal $P^\vee_{\mathfrak{m},w}(X)$
polynomial whose roots are inverse of those of $P_{\mathfrak m,x}(X)$. We then define
$\mathfrak m^\vee$ to be the maximal ideal of $\Tm_I$ so that the roots of $P^\vee_{\mathfrak{m},w}(X)$ are those
of $S_{\mathfrak m^\vee}(x)$ for every $x \in \unr(I)$.

\noindent \textit{Example}: for $x=w \in \Spl(I)$,
with the previous notations, the image $\overline{T_{w,i}}$ of $T_{w,i}$ inside  $\Tm_I/\mathfrak m$
can be written
$$\overline{T_{w,i}}=q_w^{\frac{i(1-i)}{2}} \sigma_i (\lambda_1,\cdots, \lambda_d)$$
where we write $S_{\mathfrak m}(w)=\{ \lambda_1,\cdots,\lambda_d \}$ and where the $\sigma_i$ 
are the elementary symmetric functions. Locally at $w$ the maximal ideal $\mathfrak m^\vee$
is defined by
$$T_{w,i} \in \Tm_w \mapsto q_w^{\frac{i(1-i)}{2}} \sigma_i (\lambda_1^{-1},\cdots,\lambda_d^{-1})
\in \overline \Fm_l.$$

\section{Automorphic congruences}
\label{para-principal}
\renewcommand{\thethm}{\arabic{section}.\arabic{thm}}
\renewcommand{\theequation}{\arabic{section}.\arabic{thm}}

Consider from now on a fixed place $v \in \Spl(I)$.

\begin{defin}
A $\Tm_{I}$-module $M$ is said to verify property \textbf{(P)}, if it has a finite filtration
$$(0)=\Fil^0(M) \subset \Fil^1(M) \cdots \subset \Fil^r(M)=M$$
such that for every $k=1,\cdots,r$, there exists
\begin{itemize}
\item an automorphic irreducible entire representation $\Pi_k$ of $G(\Am)$, which appears in the
cohomology of $(X_{I,\overline \eta_v})_{I \in \IC}$ with coefficients in $V_{\xi,\overline \Qm_l}$
and such that its local component $\Pi_{k,v}$ is ramified, i.e. $(\Pi_{k,v})^{GL_d(\OC_v)}=(0)$;

\item an unramified entire irreducible representation $\widetilde \Pi_{k,v}$ of $GL_d(F_v)$ with
the same cuspidal support as $\Pi_{k,v}$ and

\item a stable $\Tm_{I}$-lattice $\Gamma$ of 
$(\Pi^{\oo,v}_k)^{I^v} \otimes \widetilde \Pi_{k,v}^{GL_d(\OC_v)}$ such that
\begin{itemize}
\item either $\gr^k(M)$ is free, isomorphic to $\Gamma$,

\item or $\gr^k(M)$ is torsion and equals to some subquotient of $\Gamma/\Gamma'$ where 
$\Gamma' \subset \Gamma$ is another stable $\Tm_I$-lattice.
\end{itemize}
\end{itemize}
We will say that $\gr^k(M)$ is of type $i$ if moreover $\Pi_{k,v}$ looks like
$$\st_i(\chi) \times \chi_1 \times \cdots \times \chi_{d-i}$$
where $\chi,\chi_1,\cdots,\chi_{d-i}$ are unramified characters. When all the $\gr^\bullet(M)$
are of type $i$, then we will say that $M$ is of type $i$.
\end{defin}

\rem property \textbf{(P)} is by definition stable through extensions and subquotients: replacing
condition $\xi$-cohomological by $\xi^\vee$-cohomological, it is also stable by duality.

\begin{lemma} Consider $h \geq 1$ and $M$ an irreducible subquotient of
$$H^{d-h}_c(X^{=h}_{I,\bar s_v},V_{\xi,\overline \Qm_l}) \quad
\hbox{resp. of } H^{d-h}(X^{\geq h}_{I,\bar s_v},V_{\xi,\overline \Qm_l}),$$ 
then
\begin{itemize}
\item either $M$ verify property \textbf{(P)} and then is of type $h$ or $h+1$ (resp. of type $h$),

\item or $M$ is not a subquotient of 
$H^{d-h-1}(X^{\geq h+1}_{I,\bar s_v},V_{\xi,\overline \Qm_l})$.
\end{itemize}
\end{lemma}

\begin{proof}
The result follows from explicit computations of these $\overline \Qm_l$- cohomology groups with
infinite level given in \cite{boyer-compositio}: the reader can see a presentation of them at 
\S 3.3 (resp. \S 3.2) of \cite{boyer-aif}. Precisely for $\Pi^\oo$ an irreducible representation 
of $G(\Am^\oo)$, the isotypic component 
\begin{multline*}
\lim_{\atopp{\longrightarrow}{I \in \IC}}
H^{d-h}_c(X^{=h}_{\IC,\bar s_v},V_\xi \otimes_{\overline \Zm_l} \overline \Qm_l) \{ \Pi^{\oo,v} \}, \\
\hbox{resp. } \lim_{\atopp{\longrightarrow}{I \in \IC}}
H^{d-h}_c(X^{=h}_{\IC,\bar s_v},V_\xi \otimes_{\overline \Zm_l} \overline \Qm_l) \{ \Pi^{\oo,v} \}
\end{multline*}
is zero if $\Pi^\oo$ is not the component outside $\oo$ of an automorphic $\xi$-cohomological
representation $\Pi$. Otherwise, we distinguish three cases according to the local component
$\Pi_v$ of $\Pi$ at $v$:
\begin{itemize}
\item[(i)] $\Pi_v \simeq \st_r(\chi_v) \times \pi'_v$ with $h \leq r \leq d$,

\item[(ii)] $\Pi_v \simeq \speh_r(\chi_v) \times \pi'_v$ with $h \leq r \leq d$,

\item[(iii)] $\Pi_v$ is not of the two previous shapes,
\end{itemize}
where $\chi_v$ is an unramified character of $F_v^\times$ and $\pi'_v$ is an irreducible admissible
unramified representation of $GL_{d-h}(F_v)$. Then this isotypic component, as a
$GL_d(F_v) \times \Zm$-representation is of the following shape:
\begin{itemize}
\item in case (i) we obtain $\bigl ( \speh_h(\chi \{ \frac{h-r}{2} \} ) \times \st_{r-h}(\chi \{ \frac{h}{2} \} ) \bigr ) 
\times \pi'_v \otimes \Xi^{\frac{r-h}{2}}$ (resp. zero if $r \neq h$ and otherwise $\st_h(\chi_v)$);

\item zero in the case (ii) if $r \neq h$ and otherwise $\speh_h(\chi_v) \times \pi'_v$.

\item Finally in case (iii), the obtained $GL_d(F_v)$-representation won't have non trivial invariants
under $GL_d(\OC_v)$.
\end{itemize}
Thus taking invariants under $I$ and because $v$ doesn't divide $I$,
\begin{itemize}
\item case (i): we obtain a  $\Tm_I$-module verifying property \textbf{(P)} which is of type $h$ ou $h+1$ 
according $r=h$ or $h+1$. 

\item case (ii): the obtained $\Tm_I$-module is then not a subquotient of the cohomology group
$H^{d-h-1}(X^{\geq h+1}_{I,\bar s_v},V_{\xi,\Qm_l})$,

\item and case (iii): as it doesn't have non trivial invariants under $GL_d(\OC_v)$, we obtain nothing
else than zero.
\end{itemize}
\end{proof}

From now on we assume that there exists $i$ such that the torsion submodule of 
$H^{d-1+i}(X_{I,\bar \eta_v},V_\xi)$ is non trivial and we fix a maximal ideal $\mathfrak m$ of $\Tm_I$ 
such that the torsion of $H^{d-1+i}(X_{I,\bar \eta_v},V_\xi)_{\mathfrak m}$ is non trivial.
Let $1 \leq h \leq d$ be maximal such that there exists $i$ for which the torsion subspace 
$H^{d-h+i}(X^{\geq h}_{I,\bar s_v},V_\xi)_{\mathfrak m,tor}$ of 
$H^{d-h+i}(X^{\geq h}_{I,\bar s_v},V_\xi)_{\mathfrak m}$ 
is non reduced to zero. Notice that
\begin{itemize}
\item since the dimension of $X^{\geq d}_{I,\bar s_v}$ equals zero then we have $h<d$;

\item by the smooth base change theorem $H^\bullet(X_{I,\bar \eta_v},V_\xi) \simeq 
H^\bullet(X^{\geq 1}_{I,\bar s_v}, V_\xi)$ so that $h \geq 1$.
\end{itemize}

\begin{lemma} With the previous notations and assuming the existence of non trivial torsion cohomology classes
in $H^\bullet(X^{\geq h}_{I,\bar s_v},V_\xi)_{\mathfrak m}$, then 
$0$ is the smallest indice $i$ such that
$H^{d-h+i}(X^{\geq h}_{I,\bar s_v},V_\xi)_{\mathfrak m,tor} \neq (0)$. Moreover every irreducible non trivial
submodule of $H^{d-h}(X^{\geq h}_{I,\bar s_v},V_\xi)_{\mathfrak m,tor}$ verifies property \textbf{(P)} 
being of type $h+1$.
\end{lemma}

\begin{proof}
Consider the following short exact sequence of perverse sheaves\footnote{As we are dealing with perverse 
sheaves, note the shifts of the grading of cohomology groups.}
\addtocounter{thm}{1}
\begin{multline} \label{eq-sec00}
0 \rightarrow i_{h+1,*} V_{\xi,\overline \Zm_l,|X_{I,\bar s_v}^{\geq h+1}}[d-h-1] \longrightarrow  \\
j^{\geq h}_! j^{\geq h,*} V_{\xi,\overline \Zm_l,|X^{\geq h}_{I,\bar s_v}}[d-h] \longrightarrow 
V_{\xi,\overline \Zm_l,|X^{\geq h}_{I,\bar s_v}}[d-h] \rightarrow 0.
\end{multline}
Indeed as the strata $X^{\geq h}_{I,\bar s_v}$ are smooth and $j^{\geq h}$ is affine,
the three terms of this exact sequence are perverse and even free in the sense
of the natural torsion theory from the linear $\overline \Zm_l$-linear structure, see 
\cite{boyer-torsion} \S 1.1-1.3. 

Moreover from Artin's theorem,
see for example theorem 4.1.1 of \cite{BBD}, using the affiness of $X^{=h}_{I,\bar s_v}$, 
we deduce that 
$$H^i(X_{I,\bar s_v}^{\geq h},j^{\geq h}_! j^{\geq h,*} 
V_{\xi,\overline \Zm_l,|X^{\geq h}_{I,\bar s_v}}[d-h])$$ 
is zero for every $i<0$ 
and without torsion for $i=0$, so that for $i>0$, we have
\addtocounter{thm}{1}
\begin{equation} \label{eq-sec}
0 \rightarrow H^{-i-1}(X^{\geq h}_{I,\bar s_v},V_{\xi,\overline \Zm_l}[d-h]) \longrightarrow
H^{-i}(X^{\geq h+1}_{I,\bar s_v},V_{\xi,\overline \Zm_l}[d-h-1]) \rightarrow 0,
\end{equation}
and for $i=0$,
\addtocounter{thm}{1}
\begin{multline} \label{eq-sec0}
0 \rightarrow H^{-1}(X^{\geq h}_{I,\bar s_v},V_{\xi,\overline \Zm_l}[d-h]) \longrightarrow
H^{0}(X^{\geq h+1}_{I,\bar s_v},V_{\xi,\overline \Zm_l}[d-h-1]) \longrightarrow \\
H^0(X^{\geq h}_{I,\bar s_v},j^{\geq h}_! j^{\geq h,*} V_{\xi,\overline \Zm_l}[d-h] ) \longrightarrow
H^{0}(X^{\geq h}_{I,\bar s_v},V_{\xi,\overline \Zm_l}[d-h]) \rightarrow \cdots
\end{multline}
Thus if the torsion of $H^i(X^{\geq h}_{I,\bar s_v},V_{\xi, \overline{\mathbb Z}_l}[d-h])$ is non trivial then
$i \geq 0$ and thanks to Grothendieck-Verdier duality, the smallest such indice 
is  necessary $i=0$. Furthermore the torsion of 
$H^0(X^{\geq h}_{I,\bar s_v},V_{\xi, \overline{\mathbb Z}_l}[d-h])$ raises both into 
$H^0_c(X^{=h}_{I,\bar s_v},V_{\xi, \overline{\mathbb Z}_l}[d-h])$ and $H^0(X^{\geq h+1}_{I,\bar s_v},
V_{\xi, \overline{\mathbb Z}_l}[d-h])$, which are both free. Thus by the previous lemma, the torsion of 
$H^0(X^{\geq h}_{I,\bar s_v},V_{\xi, \overline{\mathbb Z}_l}[d-h])_{\mathfrak m}$ verifies property
\textbf{(P)} being of type $h+1$.
\end{proof}

\begin{lemma}
With previous notations, for all $1 \leq h' \leq h$, the greatest $i$ such that the torsion
of $H^{d-h'-i}(X^{\geq h'}_{I,\bar s_v},V_{\xi, \overline{\mathbb Z}_l})_{\mathfrak m,tor}$ is non zero, equals $h-h'$.
Moreover this torsion verifies property \textbf{(P)} being of type $h+1$.
\end{lemma}

\begin{proof}
We argue by induction on $h'$ from $h$ to $1$. The case $h'=h$ follows directly from
the previous lemma so that we suppose the result true up to $h'+1$ and consider
the cas of $h'$. Resume the spectral sequences (\ref{eq-sec00}) with $h'$. 
Then the result follows from (\ref{eq-sec}) and the induction hypothesis.
\end{proof}

Using the smooth base change theorem, the case $h'=1$ of the previous lemma,
then gives the following proposition.

\begin{propo} \label{prop-p1}
Let $i$ be maximal, if it exists, such that the torsion submodule of 
$H^{d-1-i}(X_{I,\bar \eta_v},V_{\xi, \overline{\mathbb Z}_l})_{\mathfrak m}$ is non zero. Then it verifies
property \textbf{(P)} being of type $i+2$.
\end{propo}

\begin{corol} \label{coro-principal}
Consider a maximal ideal $\mathfrak m$ of $\Tm_I$ and $i$ maximal, if it exists,
such that the torsion of $H^{d-1-i}(X_{I,\bar \eta},V_{\xi, \overline{\mathbb Z}_l})_{\mathfrak m}$ is non zero. 
Then there exists a set $\{ \Pi(v): ~v \in \Spl(I) \}$ of irreducible automorphic
$\xi$-cohomological representations such that
\begin{itemize}
\item for any $w \in \Spl(I)$ different from $v$, the local component at $w$ of 
$\Pi(v)$ is unramified with modulo $l$ Satake parameters given by $S_{\mathfrak m}$; 

\item the local component $\Pi(v)_v$ of $\Pi$ at $v$ is isomorphic to a representation of the
following shape
$$\st_{i+2}(\chi_{v,0}) \times \chi_{v,1} \times \cdots \times \chi_{v,d-i-2},$$ 
where $\chi_{v,0},\cdots,\chi_{v,d-i-2}$ are unramified characters of $F_w$.
\end{itemize}
\end{corol}

\rem in the first point of the previous corollary, we can of course say that for all finite places not dividing
$I \cup \bad \cup \{ l \}$, and different from $v$ in the sense of the formula (\ref{eq-facteur-v}), the modulo $l$
Satake parameters of $\Pi(v)$ and $\Pi$ are the same: but the place $v$ where the level increase must belong to $\Spl(I)$.
For $\Pi_1$ and $\Pi_2$ irreducible representations of $G(\Am)$ and $S$
the set of finite places of ramification of either $\Pi_1$ or $\Pi_2$, if the modulo $l$ Satake parameters outside $S \cup \bad$
of $\Pi_1$ and $\Pi_2$ are the same then we say that they are weakly congruent.

\begin{proof}
Consider an irreducible $\Tm_I$-submodule $M$ of $H^{d-1-i}(X_{I,\bar \eta},V_{\xi, \overline{\mathbb Z}_l})_{\mathfrak m,tor}$.
For any place $v \in \Spl(I)$, thanks to the smooth base change theorem, we have
$$H^{d-1-i}(X_{I,\bar \eta},V_{\xi, \overline{\mathbb Z}_l})_{\mathfrak m} \simeq 
H^{d-1--i}(X^{\geq 1}_{I,\bar s_v},V_{\xi, \overline{\mathbb Z}_l})_{\mathfrak m}.$$
From the previous proposition, this module $M$ verifies property \textbf{(P)} 
being of type $i+2$ so that it exists an automorphic irreducible $\xi$-cohomological
representation $\Pi(v)$ verifying the required properties.
\end{proof}

\section{Completed cohomology and torsion classes}
\renewcommand{\thethm}{\arabic{section}.\arabic{thm}}
\renewcommand{\theequation}{\arabic{section}.\arabic{thm}}
\label{para-torsion}

Given a level $I^l \in \IC$ maximal at $l$, recall that the completed cohomology groups are 
$$\widetilde{H}^i_{I^l}(V_{\xi,\OC/\lambda^n}):=\lim_{\atopp{\longrightarrow}{I_l}} 
H^i(X_{I^lI_l},V_{\xi,\OC/\lambda^n})$$
and
$$\widetilde{H}^i_{I^l}(V_{\xi,\OC}):=\lim_{\atopp{\longleftarrow}{n}} \widetilde{H^i}_{I^l}
(V_{\xi,\OC/\lambda^n}),$$
where $\OC$ is the ring of integers of a finite extension of $\Qm_l$ on which the representation $\xi$
is defined.

\begin{nota} When $\xi=1$ is the trivial representation, we will denote
$$\widetilde{H}^i_{I^l}:=\widetilde{H}^i_{I^l}(V_{1,\OC}) \otimes_\OC \overline \Zm_l.$$
\end{nota}

\rem for $n$ fixed, there exists an open compact subgroup $I_l(n)$ such that, using the notations below
\ref{nota-Vxi}, every $I_l \subset I_l(n)$ acts trivialy on
$W_{\xi,\OC} \otimes_\OC \OC/\lambda^n$. We then deduce that the completed cohomology groups 
don't depend of the choice of $\xi$ in the sense where, see theorem 2.2.17 of \cite{emerton-invent}:
\addtocounter{thm}{1}
\begin{equation} \label{eq-completee0}
\widetilde{H}^i_{I^l}(V_{\xi,\OC})  \otimes_\OC \overline \Zm_l \simeq \widetilde{H}^i_{I^l}\otimes W_\xi
\end{equation}
where $G(\Qm_l)$ acts diagonally on the right side.

Scholze, see \cite{scholze-torsion} proposition IV.2.2, has showed that the 
$\widetilde{H}^i_{I^l}(V_{\xi,\OC})$ are trivial for all $i>d-1$. In our situation we can prove that for all
divisor $s$ of $d$, there are non zero for $i=d-s$: the argument is quite simple but it uses some
particular results about entire notions of intermediate extension of Harris-Taylor's local systems.
As we don't really need such precision, we only prove the following property.

\begin{propo} \label{prop-completee-nonnul}
For each divisor $s$ of $d=sg$ and for a level $I^l$ outside $l$ small enough, there exists 
$i \leq d-s$ such that 
$\widetilde{H}^{i}_{I^l} \otimes_{\overline \Zm_l} \overline \Qm_l$ has, as a 
$GL_d(F_v)$-representation, an irreducible quotient with degeneracy depth equals to $s$.
\end{propo}

\begin{proof}
Recall the Hochschild-Serre spectral sequence allowing to compute the cohomology at finite
level from completed one
\addtocounter{thm}{1}
\begin{equation} \label{eq-ss-completee}
E_2^{i,j}=H^i(I_l, \widetilde H^j_{I_l} \otimes V_{\xi, \overline{\mathbb Z}_l}) \Rightarrow H^{i+j}(X_{I^lI_l},
V_{\xi, \overline{\mathbb Z}_l}).
\end{equation}
Let $v \in \Spl(I^l)$ be a fixed place over some prime number $p \neq l$.
Consider then a divisor $s$ of $d=sg$ and an automorphic representation $\Pi$ which is
cohomological relatively to a algebraic representation $\xi$ of $G$ and such that its local component
at the place $v$ is isomorphic to $\speh_s(\pi_v)$ where $\pi_v$ is an irreducible cuspidal representation
of $GL_g(F_v)$. As before we choose a finite level $I^l$ outside $l$ so that $\Pi$ has
non trivial invariant vectors under $I^l$. According to \cite{boyer-compositio}, the
$\Pi^\oo$-isotypic factor of the $\overline \Qm_l$-cohomology group of indice $d-s$ is non trivial
for $I=I^lI_l$ with $I_l$ small enough. The result then follows from the spectral sequence
(\ref{eq-ss-completee}).
\end{proof}

Let  $\widehat{H}^i_{I^l}(V_{\xi,\OC})$ be the $p$-adic completion of $H^i_{I^l}(V_{\xi,\OC}):=
{\displaystyle \lim_{\atopp{\longrightarrow}{I_l}} } H^i(X_{I^lI_l},V_{\xi,\OC})$ that is
$$\widehat{H}^i_{I^l}(V_{\xi,\OC})=\lim_{\atopp{\longleftarrow}{n}} \Bigl (
\lim_{\atopp{\longrightarrow}{I_l}} H^i(X_{I^lI_l},V_{\xi,\OC}[d-1]) ~/~ \lambda^n  H^i(X_{I^lI_l},V_{\xi,\OC}[d-1]) \Bigr ).$$
It kills the $p$-divisible part of $H^i_{I^l}(V_{\xi,\OC})$. Consider also
the $p$-adic Tate module of $H^i_{I^l}(V_{\xi,\OC})$
$$T_pH^i_{I^l}(V_{\xi,\OC}):= \lim_{\atopp{\longleftarrow}{n}} H^i_{I^l}(V_{\xi,\OC}) [\lambda^n].$$
whom knows only about torsion. Recall then the short exact sequence
\addtocounter{thm}{1}
\begin{equation}\label{eq-sec-fond} 
0 \rightarrow \widehat{H}^i_{I^l}(V_{\xi,\OC}) \longrightarrow \widetilde{H}^i_{I^l}(V_{\xi,\OC}) \longrightarrow 
T_pH^{i+1}_{I^l}(V_{\xi,\OC}) \rightarrow 0.
\end{equation}

When $\xi$ is a regular algebraic representation, the cohomology of $X_I$ with coefficients in $V_{\xi, \overline{\mathbb Q}_l}$, 
is concentrated in middle degree
and so $\widehat{H}^i_{I^l}(V_{\xi,\overline{\mathbb Q}_l})$ is trivial for all $i \neq d-1$. Let $s \geq 2$ be a divisor of $d$ and $i \leq d-s<d-1$
such that, thanks to the previous proposition, for some finite level $I^l$ outside $l$ small enough
$$\widetilde{H}^i_{I^l}(V_{\xi, \overline{\mathbb Z}_l}) \otimes_{ \overline{\mathbb Z}_l} \overline{\mathbb Q}_l \simeq 
T_pH^{i+1}_{I^l}(V_{\xi, \overline{\mathbb Z}_l})$$
is non trivial. It means then that for all $n \geq 1$, there exists an open compact subgroup $I_l$ small enough, 
for which $H^{i+1}(X_{I^lI_l},V_{\xi, \overline{\mathbb Z}_l})$ has a class of exactly $\lambda^n$-torsion so that through
the process of completed cohomology, i.e. when you first take the limit on $I_l$ and then on $\lambda^n$,
the reductions modulo $\lambda^n$ for varying $n$, of these torsion classes give torsion free classes generating an automorphic
representation $\Pi$ with depth of degeneracy equals to $s$ and trivial weight.
From proposition \ref{prop-completee-nonnul}, for $\mathfrak m$ a maximal ideal of
$\Tm_I$ associated to this $\Pi$, there exists a set $\{ \Pi(v); ~ v \in \Spl(I) \}$ such that the properties
of corollary \ref{coro-principal} hold:
\begin{itemize}
\item in particular these $\Pi(v)$ are tempered representations, non isomorphic and weakly congruent in twos; 
there are all of the same regular weight;

\item each of these tempered irreducible representations $\Pi(v)$ of regular weight $\xi$ is also weakly 
congruent with $\Pi$,
an irreducible representation of trivial weight with degeneracy depth $>1$.

\end{itemize}
More generally using the isomorphism (\ref{eq-completee0}) for any $\xi$ not necessarily trivial or regular, take as before 
$i\leq d-s$ minimal such
that the free quotient of $\widetilde{H}^{i}_{I^l}$ has an irreducible quotient $\Pi$ with depth of degeneracy $s$ and let 
$\mathfrak m$ be the maximal ideal of $\Tm_{I^l}$ associated to such a $\Pi$.
Thus for any irreducible algebraic representation $\xi$ 
we have again $\widetilde{H}^{i}_{I^l}(V_{\xi,\OC})_{\mathfrak m}  \neq (0)$ so that for every 
$I_l$ small enough and for all $n$, we have 
$H^{i}(X_{I^lI_l},V_{\xi,\OC/(\lambda^n)})_{\mathfrak m} \neq (0)$ so, using (\ref{eq-sec-fond}),
\begin{itemize}
\item[(i)] either $\widehat{H}^i_{I^l}(V_{\xi,\OC})$ is non trivial so that there exists
a $\xi$-cohomological automorphic representation $\Pi'$ which is weakly congruent
with $\Pi$: note that, by minimality of $i$, 
such a $\Pi'$ is necessary of degeneracy depth $s$;

\item[(ii)] or, $T_pH^{i+1}_{I^l}(V_{\xi,\OC}) \neq (0)$ so thanks to 
proposition \ref{prop-p1} there exists an $\xi$-cohomo\-logical automorphic tempered representation
$\Pi'$ whose modulo $l$ Satake parameters at places of $\unr(I) \backslash \{ v \}$ are given by 
$\mathfrak m$. 

\end{itemize}
So for any weights $\xi_1 \neq \xi_2$, we can construct weakly automorphic congruences between representations $\Pi'_1$
and $\Pi'_2$ respectively of weight $\xi_1$ and $\xi_2$, governed by a maximal ideal $\mathfrak m$ attached to some 
irreducible automorphic representation, cohomological for the trivial character, with degeneracy depth $s \geq 2$.
Concerning the degeneracy depth of $\Pi'_1$ and $\Pi'_2$ it is equal to $s$ or $1$ according to they fall in case (i) or (ii).

\noindent \textit{Example}: consider the most trivial case where $\xi_1$ is the trivial character and $\xi_2$ is a regular one and
take $\mathfrak m$ the maximal ideal associated to the trivial representation $\Pi$ of $G(\Am)$ with degeneracy depth $s=d$ which
is trivially, as $\widetilde H^{0}_{I^l} \simeq \Zm_l$, $\xi_1$-cohomological. As $\xi_2$ is regular we are then in case (ii), that is
for any deep enough $I_l$, the torsion of $H^{1}(X_{I^lI_l},V_{\xi,\OC})$ is non trivial. 
Thanks to proposition \ref{prop-p1}, this torsion raises
in characteristic zero to a tempered automorphic representation $\xi$-cohomological $\Pi$ such that:
\begin{itemize}
\item its local component at $v$ is a Steinberg representation $\st_d(\chi_v)$ with $\chi_v$ congruent
to the trivial character modulo $l$,

\item its local Satake parameters at every place of $\unr(I) \backslash \{ v \}$ are those of the trivial
character of $G(\Am)$.
\end{itemize}

\section*{Acknowledgements}
  %%acknowledgements here%%

The author would like to thanks the referee for his numerous remarks which clearly contribute to improve the first version of this paper. 
From the very beginning, I've benefit from Michael Harris' interest to my work and I take this opportunity to thanks him again.

\bibliographystyle{mrl}

\def\cftil#1{\ifmmode\setbox7\hbox{$\accent"5E#1$}\else
  \setbox7\hbox{\accent"5E#1}\penalty 10000\relax\fi\raise 1\ht7
  \hbox{\lower1.15ex\hbox to 1\wd7{\hss\accent"7E\hss}}\penalty 10000
  \hskip-1\wd7\penalty 10000\box7} \def\cprime{$'$}

%\bibliography{bib-ok}

\end{document}